\documentclass{amsart}
\usepackage[utf8]{inputenc}
\usepackage{amsfonts, amsmath, amsthm, amssymb,tikz-cd}
\usepackage{graphicx,color}

\author{Trung Hoa Dinh, Cong Trinh Le, The Van Nguyen, Bich Khue Vo}

\address{Trung Hoa Dinh, Department of Mathematics, Troy University, Troy, AL 36082, United States\\}
\email{thdinh@troy.edu}

\address{The Van Nguyen, Faculty of Mathematics-Mechanics-Informatics, VNU University of Science, Vietnam National University, Hanoi, Vietnam\\ }
\email{nguyenvanthe@hus.edu.vn}

\address{Cong Trinh Le, Division of Computational Mathematics and Engineering, Institute for Computational Science, Ton Duc Thang University, Ho Chi Minh City, Vietnam\\
Faculty of Mathematics and Statistics, Ton Duc Thang University, Ho Chi Minh City, Vietnam\\ }
\email{lecongtrinh@tdtu.edu.vn}

\address{Bich Khue Vo, Faculty of Economics and Law, University of Finance and Marketing, Ho Chi Minh City, Vietnam\\}
\email{votbkhue@gmail.com}

\title[Matrix power means and new characterizations of operator monotone functions ]{Matrix power means and new characterizations of operator monotone functions }


\theoremstyle{plain}
\newtheorem{theorem}{Theorem}
\newtheorem{lemma}[theorem]{Lemma}
\newtheorem{corollary}[theorem]{Corollary}

\newtheorem{proposition}[theorem]{Proposition}
\def\Tr{{\rm Tr\,}}
\def\M{{\mathbb{M}\,}}
\theoremstyle{definition}

\begin{document}
\maketitle

\begin{abstract}
For positive definite matrices $A$ and $B$, the Kubo-Ando matrix power mean is defined as
$$
P_\mu(p, A, B) = A^{1/2}\left(\frac{1+(A^{-1/2}BA^{-1/2})^p}{2}\right )^{1/p} A^{1/2}\quad  (p \ge 0).
$$

In this paper, for $0\le p \le 1 \le q$, we show that if one of the following inequalities 
\begin{align*} 
f(P_\mu(p, A, B)) \le f(P_\mu(1, A, B)) \le f(P_\mu(q, A, B))\nonumber
\end{align*}
holds for any positive definite matrices $A$ and $B$, then the function $f$ is operator monotone on $(0, \infty).$ We also study the inverse problem for non-Kubo-Ando matrix power means with the powers $1/2$ and $2$. As a consequence, we establish new charaterizations of operator monotone functions with the non-Kubo-Ando matrix power means. \end{abstract} 

\section{Introduction}

It is well-known that for $0 \le p \le 1 \le q$ and for non-negative numbers $a$ and $b$,
$$
\sqrt{ab} \le 
\left(\frac{a^p+b^p}{2}\right )^{1/p} \le \frac{a+b}{2} \le \left(\frac{a^q+b^q}{2}\right )^{1/q},
$$
or, 
$$
\sqrt{ab} \le \mu(p, a, b) \le \mu(1, a, b)  \le \mu(q, a, b),
$$
where $\mu(p, a, b) = \left(\frac{a^p+b^p}{2}\right )^{1/p}.$ Let $f$ be a continuous increasingly monotone function on $[0, \infty)$. Then
\begin{equation}\label{scalar}
   f( \sqrt{ab}) \le 
f\left(\left(\frac{a^p+b^p}{2}\right )^{1/p}\right ) \le f\left(\frac{a+b}{2}\right ) \le f\left(\left(\frac{a^q+b^q}{2}\right )^{1/q}\right ).
\end{equation}

Suppose that one of inequalities in (\ref{scalar}) holds for any non-negative numbers $a$ and $b$. Then, the following question is natural: Is it true that $f$ is increasingly monotone on $[0, \infty)$? 

Mention that if the inequality $$f(\sqrt{ab}) \le f\left(\frac{a+b}{2}\right)$$ holds for any $a, b \ge 0$, then $f$ is increasingly monotone on $[0, \infty)$. In \cite{hoa2015} the first author considered the following reverse AGM inequality
\begin{equation}\label{ragm}
\frac{a+b-|a-b|}{2} \le \sqrt{ab},
\end{equation}
and show that this inequality also characterizes increasingly monotone function. Some other characterizations were also obtained by the first author and his co-authors in \cite{hoa2018}. 

%\bigskip 

Now, let $\M_n$ be the algebra of $n\times n$ matrices over $\mathbb{C}$ and $\mathcal{P}_n$ denote the cone of positive definite elements in $\M_n$. Denote by $I_n$ the identity matrix of $\M_n$. A continuous function $f$ is said to be \textit{operator monotone} on $I \subset \mathbb{R}$ if for any Hermitian matrices  $A$ and $B$ with spectra in $I,$
$$
A \le B \quad \implies \quad f(A) \le f(B),
$$
where $f(A)$ is understood by means of the functional calculus. Operator monotone functions were firstly introduced by Loewner in 1930 \cite{loewner}.  In 1980, Kubo and Ando \cite{kuboando} introduced the theory of operator means on the set of $B(H)^+ \times B(H)^+$, where $B(H)^+$ is the set of positive invertible operators in a Hilbert space $H$. The main result in their paper is the one-to-one correspondence between operator means $\sigma$ and operator monotone functions $f$ on $(0, \infty)$ defined by
\begin{equation}\label{1-1}
A\sigma B = A^{1/2}f(A^{-1/2}BA^{-1/2})A^{1/2}.
\end{equation}
In 1996, Petz \cite{petz} introduced the theory of monotone metric in quantum information theory which was based on operator monotone functions. Therefore, such functions are important matrix analysis, quantum information and other areas as well. The authors refer readers to the books of  William Donoghue \cite{donoghue} and Barry Simon \cite{ simon} for more details about operator monotone functions. 

One of the most important means is the geometric mean $$A\sharp B = A^{1/2}(A^{-1/2}BA^{-1/2})^{1/2}A^{1/2}$$ which was firstly defined by Pusz and Woronowicz \cite{pusz}. The matrix AGM inequality states that for positive definite matrices $A$ and $B,$
\begin{equation}\label{agm}
A\sharp B \le \frac{A+B}{2}.
\end{equation} In 2014, Hiai and Ando gave a new characterization of operator monotone functions using (\ref{agm}). 
 They showed that if 
$$
f(A\sharp B) \le f\left (\frac{A+B}{2} \right)
$$
for any positive definite matrices $A$ and $B$, then $f$ is operator monotone on $(0,\infty)$. In 2015, the first author obtained a new characterization based on the matrix version of (\ref{ragm}): 
$$
f\left(\frac{A+B-|A-B|}{2}\right) \le f(A\sharp B)
$$
whenever $A$, $B$ are positive definite matrices. In another papers \cite{hoaosakatomiyama, hoa2018} a series of new characterizations related to the matrix Heron mean and Powers-Stormer inequality in quantum hypothesis testing theory \cite{au} were established.

Now, let $p$ be a real number, and  $a, b$ be positive. According to the relation (\ref{1-1}) the power mean $\mu(p, a, b)=\big(\frac{a^p+b^p}{2})^{1/p}$ is  corresponding to the monotone function $f$ as
$$
f_{\mu,p}(t) = \left(\frac{1+t^p}{2}\right )^{1/p}.
$$
Then for positive definite matrices $A$ and $B$, the Kubo-Ando matrix power mean is defined as
$$
P_\mu(p, A, B) = A^{1/2}\left(\frac{1+(A^{-1/2}BA^{-1/2})^p}{2}\right )^{1/p} A^{1/2}.
$$
Therefore, the chain of inequalities  (\ref{scalar}) for matrices looks like
\begin{align}\label{matrix}
f(A\sharp B) & \le f(A^{1/2}f_{\mu, p}(A^{-1/2}BA^{-1/2})A^{1/2}) \\
& \le f(A^{1/2}f_{\mu, 1}(A^{-1/2}BA^{-1/2})A^{1/2}) \nonumber \\
& \le f(A^{1/2}f_{\mu, q}(A^{-1/2}BA^{-1/2})A^{1/2}).\nonumber
\end{align}

Motivated by works mentioned above, in this paper we investigate new characterizations of operator monotone functions by  inequalities in  (\ref{matrix}). We show that if one of the inequalities in (\ref{matrix}) holds for any positive finite matrices $A$ and $B$, then the function $f$ is operator monotone on $(0, \infty).$ 

The more difficult situation is for the naive matrix extension of the power means. Let $1/2  \le p \le 1 \le q$. The function $t^{1/q}$ is operator concave, while the function $t^{1/p}$ is operator convex. Then we have
\begin{equation}\label{matrix3}
\left(\frac{A^p+B^p}{2}\right )^{1/p} \le \frac{A+B}{2} \le \left(\frac{A^q+B^q}{2}\right )^{1/q}
\end{equation}
whenever $A$ and $B$ are positive semidefinite. It is worth noting that the inequalities in (\ref{matrix3}) were discussed by Audenaert and Hiai and \cite{hiaiaudenaert}, where they obtained conditions on $p$ and $q$ for such that (\ref{matrix3}) holds true. In \cite{hoakhue} the authors also studied a new type of operator convex functions.

%Now, if $f$ is an operator monotone function on $[0, \infty)$, then from (\ref{matrix3}) for any positive semidefinite matrices $A$ and $B$ we have
%\begin{equation}\label{matrix2}
%f\left(\left(\frac{A^p+B^p}{2}\right )^{1/p}\right ) \le f\left(\frac{A+B}{2}\right ) \le f\left(\left(\frac{A^q+B^q}{2}\right )^{1/q}\right).
%\end{equation}

In Section \ref{section3} we study the inverse problem for the non-Kubo-Ando matrix power means when $q=1/2$ and $p=2$. As a consequence, we establish a new characterization of operator monotone functions by inequalities in (\ref{matrix3}). 

\section{Kubo-Ando matrix power means and characterizations}

In this section we study the problem of characterization of operator monotone functions using Kubo-Ando matrix power means. We start with the scalar cases.

\begin{theorem}\label{Scalar.Theorem}
Let $f$ be a continuous function on $[0, \infty)$. For $0\leq p \leq 1 \leq q$, suppose that one of the  following inequalities holds for all any non-negative numbers $a \leq b$:
\begin{align}
&f(\sqrt{ab}) \leq f\left(\left(\dfrac{ a^p+b^p}{2}\right)^{1/p}\right), \label{scalar.ineq1} \\
&f\left(\left(\dfrac{a^p+b^p}{2}\right)^{1/p} \right)\leq f\left(\dfrac{a+b}{2}\right), \label{scalar.ineq2}\\
&f\left(\dfrac{a+b}{2}\right) \leq f\left(\left(\dfrac{a^q+b^q}{2}\right)^{1/q} \right).
\label{scalar.ineq3}
\end{align}
Then the function $f$ is increasingly monotone on $[0,\infty)$.
\end{theorem}

\begin{proof}
To prove the theorem,  we need to show that $f(x) \leq f(y)$ for any $0 \leq x \leq y$. Since all three inequalities in assumption are homogeneous, it suffices to prove the theorem in the case when either $x = 1$ and $y \geq 1$ or $x \leq 1$ and $y = 1$.  

Suppose that the first inequality \eqref{scalar.ineq1} holds, we have to show that for any $y \geq 1$ there exist $a,b \geq 0$ such that $1 = \sqrt{ab}$ and $ y = \left(\dfrac{a^p+b^p}{2}\right)^{1/p}$. Or, equivalently, there exists $a > 0$ such that 
\begin{equation} \label{eq.1}
y = \left( \dfrac{a^p+a^{-p}}{2}\right)^{1/p}.
\end{equation} 
Note that the function $f(x) = 2^{-1/p}(x^p+x^{-p})^{1/p}$ is surjective from $(0,\infty)$ onto $[1, \infty)$. Therefore, for any $y \geq 1$, there exist $a > 0$ such that the identity \eqref{eq.1} holds. 

Now, suppose that the second inequality \eqref{scalar.ineq2} holds, we will show that for any $x \leq 1$, there exist $a, b \geq 0$ such $1 = (a+b)/2$ and  $x = \left(\dfrac{a^p+b^p}{2}\right)^{1/p}.$ Or, equivalently, there exists $0 \leq a \leq 2$ such that 
\begin{equation}\label{eq.2}
x = \left(\dfrac{a^p+(2-a)^p}{2}\right)^{1/p}. 
\end{equation}
The function $g(x) = 2^{-1/p}(x^p+(2-x)^p)^{1/p}$ is surjective from $[0,2]$ onto $[2^{1-1/p}, 1]$. Therefore, for any $\gamma \leq x \leq 1$ where $\gamma = 2^{1-1/p}$, there exists $a>0$ such that the identity \eqref{eq.2} holds. Consequently, if $ \gamma y \leq x \leq y$, from the homogeneous property of the second inequality it implies that 
\[ f(x) \leq f(y).\] 
If $x < \gamma y$, let $\gamma_0 \in (\gamma,1)$ and consider the sequence $\{\gamma_0^n\}_{n \in \mathbb{N}}$. Since $\gamma_0^n \to 0$ as $n \to \infty$, there exists $k \in \mathbb{N}$ such that 
\[ 0 < \gamma \left(\gamma_0^n y\right) < \gamma_0^{n+1} y < x \leq \gamma_0^n y \leq \gamma_0^{n-1} y \leq \dots \leq \gamma_0 y \leq y.\]
Hence, from the previous argument it implies that 
\[ f(x) \leq f\left(\gamma_0^n y\right) \leq \dots \leq f\left(\gamma_0 y\right) \leq f(y).\]
Therefore, $f$ is increasingly monotone on $[0,\infty)$ if the second inequality holds. 

For the case when the last inequality \eqref{scalar.ineq3} holds, we proceed similarly as the case of the second inequality, in which we work for $x=1, y \geq 1$, and $ \gamma = 2^{1-1/q} > 1$. 
\end{proof}

\begin{theorem}\label{Matrix.Theorem 1}
Let $0 < p \leq 1$, and $f$ be a continuous function on $[0,\infty)$ that satisfies the following inequality
\begin{align}
f\left( A \sharp B\right) &\leq f\left(A^{1/2} f_{\mu, p}\left(A^{-1/2}BA^{-1/2}\right)A^{1/2}\right), \label{eq1.matrix.theo} 
\end{align}
for any positive definite matrices $A$ and $B$. Then $f$ is operator monotone on $(0, \infty)$.
\end{theorem}

\begin{proof}
Suppose that the inequality \eqref{eq1.matrix.theo} holds, it suffices to show that for any $0 \leq X \leq Y$, there exist two positive semidefinite matrices $A$ and $B$ such that 
\begin{align*}
X = A \sharp B, \quad 
Y = A^{1/2}f_{\mu, p}\left(A^{-1/2}BA^{-1/2}\right)A^{1/2}.
\end{align*}
Firstly, let us consider the case when $X = I_n$. We now show that there exist positive definite matrices $A_0$ and $B_0$ such that $A_0 \sharp B_0 = I_n$ and 
\begin{equation}\label{eq1.pro.of matrix theo}
Y = A_0^{1/2} f_{\mu, p}\left(A_0^{-1/2}B_0A^{-1/2}\right)A_0^{1/2} = A_0 f_{\mu,p}\left(A_0^{-2}\right).
\end{equation}
Since the function $h(x) = x f_{\mu,p}(x^{-2}) = \left( \dfrac{x^p+x^{-p}}{2}\right)^{1/p}$ is surjective from $(0,\infty)$ to $[1,\infty)$, we obtain that for any $I_n \leq Y$ there exist a matrix $A_0 > 0$ satisfying \eqref{eq1.pro.of matrix theo}. The matrix $B_0$ is equal to $A_0^{-1}$.  

In general, for $0 < X \leq Y$ we have $I_n \leq X^{-1/2}YX^{-1/2}$. By the above arguments, we can find positive semidefinite matrices $A_0$ and $B_0$ such that $A_0 \sharp B_0 = I_n$ and 
\[ X^{-1/2} Y X^{-1/2} = A_0^{1/2} f_{\mu, p}\left(A_0^{-1/2}B_0A^{-1/2}\right)A_0^{1/2} = P_{\mu}\left(p,A_0,B_0\right).\]
Consequently, applying \eqref{eq1.matrix.theo} to matrices $A = X^{1/2}A_0X^{1/2}$ and $B = X^{1/2}B_0 X^{1/2}$, we obtain that $f(X) \leq f(Y)$. In other words, $f$ is operator monotone. 
\end{proof}

\begin{theorem}\label{Matrix.Theo 2}
Let $0 < p \leq 1 \leq q$, and $f$ be a continuous function on $[0,\infty)$ that satisfies one of the following inequalities
\begin{align}
&f\left(\dfrac{A+B}{2}\right) \leq f\left(A^{1/2} f_{\mu, q}\left(A^{-1/2}BA^{-1/2}\right)A^{1/2}\right), \label{eq3.matrix.theo} \\ 
&f\left(A^{1/2} f_{\mu, p}\left(A^{-1/2}BA^{-1/2}\right)A^{1/2}\right) \leq f\left(\dfrac{A+B}{2}\right), \label{eq2.matrix.theo} \end{align}
for any positive definite matrices $A$ and $B$. Then $f$ is operator monotone on $(0, \infty)$.
\end{theorem}

To prove Theorem \ref{Matrix.Theo 2}, we will need the following lemma.
\begin{lemma}\label{Lemma.TB}
Suppose $X$ and $Y$ are positive definite matrices satisfying $ X \leq Y < \gamma X$ (resp. $\gamma X < Y \leq X$) where $\gamma = 2^{1-1/q}$ (resp. $\gamma = 2^{1-1/p}$), then there exist positive matrices $A$ and $B$ such that 
\[ X = \frac{A +  B}{2}  \quad   \text{ and } \quad Y = P_{\mu}\left(q,A,B\right) \, (resp. \,\, Y = P_{\mu}\left(p,A,B\right)),\]
where $0 < p \leq 1 \leq q$. 
\end{lemma}

\begin{proof}
We show the lemma when $X \leq Y < \gamma X$, the remaining case can be obtained similarly. Firstly, let us consider the case when $X = I_n$, it suffices to show that given $I_n \leq Y = U \mathrm{diag}\left(\left\{ \lambda_i\left(Y\right)\right\}\right) U^* \leq \gamma I_n,$ we can find $A_0, B_0 \geq 0$ such that 
\[ I_n = \frac{A_0+  B_0}{2}  \quad \text{ and } \quad Y = P_{\mu}(q,A_0,B_0).\]
Or, equivalently, there exists $0 < A_0 \leq 2I_n$ such that $Y = P_{\mu}\left(q,A_0, 2 - A_0\right) = \varphi(A_0),$ where
\[ \varphi(x) = x^{1/2}f_{\mu,q}\left(x^{-1/2}(2-x)x^{-1/2}\right)x^{-1/2} = \left(\dfrac{x^q+(2-x)^q}{2}\right)^{1/q}.  \]
Note that $\varphi$ is continuous on $[0,2]$ and surjective from $[0,2]$ onto $[1,\gamma]$. Since $\lambda_{i}(Y) \in [1,\gamma],$ for each $i$, we can choose $\delta_i(Y) \in [0,2]$ such that $\varphi\left(\delta_i(Y)\right) = \lambda_i(Y).$ The matrix 
\[ A_0 := U \mathrm{diag}\left(\left\{\delta_i(Y)\right\}\right) U^*\]
satisfies 
\[ \varphi(A_0) = U \mathrm{diag}\left(\left\{\varphi\left(\delta_i(Y)\right)\right\}\right) U^* = U \mathrm{diag}\left(\{\lambda_i(Y)\}\right) U^*.\]

In general, for $0 < X \leq Y < \gamma X$, we have $I_n \leq X^{-1/2}YX^{-1/2} < \gamma I_n$. By the above arguments, we can find positive definite matrices $A_0$ and $B_0$ such that $(A_0 + B_0)/2 = I_n$ and $X^{-1/2}YX^{-1/2} = P_{\mu}(q,A_0,B_0)$. Now, let $A = X^{1/2}A_0X^{1/2}$ and $B = X^{1/2}B_0X^{1/2}$, we have $(A+ B)/2 = X$ and 
\[ Y = X^{1/2}P_{\mu}(q,A_0,B_0)X^{1/2} = P_{\mu}(q,A,B),\]
which completes the proof of Lemma \ref{Lemma.TB}.
\end{proof}
We are now ready to prove Theorem \ref{Matrix.Theo 2}. 
\begin{proof}[\bf Proof of Theorem \ref{Matrix.Theo 2}]
First we prove the case when the inequality \eqref{eq3.matrix.theo}. Let $0 \leq X \leq Y$ and $Y_0 = X^{-1/2}YX^{-1/2}$, and choose $\gamma_0 \in (1,2^{1-1/q})$. Consider the spectral decomposition, $Y_0 = \sum_{i = 1}^r \lambda_i E_i$ with the eigenvalues $\lambda_i$ listed in the decreasing order. Then, there exists a set of non-ascending integers $\{m_i \,|\, 1 \leq i \leq r\}$ such that 
\[ \gamma_0^{m_i} < \lambda_i \leq \gamma_0^{m_i+1}.\]  
Let $\ell_1 < \ell_2 < \dots < \ell_t = r$ be the sequence of indexes such that 
\begin{align*}
m_1 = \dots = m_{\ell_1} > m_{\ell_1 + 1}= \dots = m_{\ell_2} > m_{\ell_2 + 1} = \dots = m_{\ell_3} > \dots > m_{\ell_{t-1}+1} = \dots = m_{\ell_t} = m_r.
\end{align*} 
We have 
\[ \gamma_0^{m_{\ell_t}} < \lambda_r < \dots < \lambda_{\ell_{t-1}+1} \leq \gamma_0^{m_{\ell_t}+1} \leq \gamma_0^{\ell_{t-1}} < \lambda_{\ell_{t-1}} < \dots<\lambda_{\ell_{t-2}+1} \leq \gamma_0^{\ell_{t-1}+1} < \dots < \lambda_1 \leq \gamma_0^{\ell_1 + 1}. 
\]
It follows that
\begin{align*}
&I < \gamma_0 I < \gamma_0^2 I < \dots < \gamma_0^{m_{\ell_t}}I = \gamma_0^{m_{\ell_t}}\left(E_1+E_2+\dots+E_r\right)  \\
&\leq \lambda_rE_r + \dots +\lambda_{\ell_{t-1}+1}E_{\ell_{t-1}+1} + \gamma_0^{m_{\ell_{t}}} \left( E_{\ell_{t-1}} + E_{\ell_{t-1}-1} + \dots + E_1\right) \\
&\leq \lambda_rE_r + \dots +\lambda_{\ell_{t-1}+1}E_{\ell_{t-1}+1} + \gamma_0^{m_{\ell_{t}}+1} \left( E_{\ell_{t-1}} + E_{\ell_{t-1}-1} + \dots + E_1\right) \\
&\dots \\
&\leq \lambda_rE_r + \dots +\lambda_{\ell_{t-1}+1}E_{\ell_{t-1}+1} + \gamma_0^{m_{\ell_{t-1}}} \left( E_{\ell_{t-1}} + E_{\ell_{t-1}-1} + \dots + E_1\right)\\
&\leq \lambda_rE_r+\dots + \lambda_{\ell_{t-2}+1}E_{\ell_{t-2}+1} + \gamma_0^{m_{\ell_{t-1}}}\left(E_{\ell_{t-2}}+E_{\ell_{t-2}-1}+\dots+E_{1}\right) \\
&\leq \lambda_rE_r+\dots + \lambda_{\ell_{t-2}+1}E_{\ell_{t-2}+1} + \gamma_0^{m_{\ell_{t-1}}+1}\left(E_{\ell_{t-2}}+E_{\ell_{t-2}-1}+\dots+E_{1}\right) \\
&\dots \\
&\leq \lambda_r E_r + \dots +\lambda_{\ell_2+1}E_{\ell_2+1} + \gamma_0^{m_{\ell_2}}\left(E_{\ell_1} + E_{\ell_1-1} + \dots + E_1\right)\\
&\leq \lambda_r E_r + \dots +\lambda_{\ell_2+1}E_{\ell_2+1} + \gamma_0^{m_{\ell_2}+1}\left(E_{\ell_1} + E_{\ell_1-1} + \dots + E_1\right)\\
&\dots \\
&\leq \lambda_r E_r + \dots +\lambda_{\ell_2+1}E_{\ell_2+1} + \gamma_0^{m_{\ell_1}}\left(E_{\ell_1} + E_{\ell_1-1} + \dots + E_1\right) \\
&\leq \lambda_r E_r + \dots \lambda_1E_1 = Y_0 \leq \gamma_0^{m_{\ell_1}+1} I.
\end{align*}
After multiplying each term of the chain of inequalities on both sides by $X^{1/2}$, let $Z_k$ be the $k$-th expression of the chain, we obtain the following chain inequalities
\[ 0 \leq X = Z_1 \leq Z_2\leq \dots \leq Z_{m-1}\leq Y \leq Z_{m} = \gamma_0^{m_{\ell_1}+1}X. \]
where $m$ is a positive integer. The previous calculation gives us $Z_k \leq Z_{k+1} \leq \gamma Z_k$. Therefore, it follows from Lemma \ref{Lemma.TB} that there exist positive definite matrices $A$ and $B$ such that 
\[ Z_k = (A + B)/2  \quad \text{and} \quad  Z_{k+1} = P_{\mu}(q,A,B).\]
Consequently, 
\[ f(X) \leq f(Z_1) \leq f(Z_2) \leq \dots \leq f(Z_{m-1}) \leq f(Y).\]
In other words, $f$ is operator monotone. 

The proof in the case when \eqref{eq2.matrix.theo} holds is similar.
\end{proof}
\section{The inverse problem for non-Kubo-Ando matrix power means}\label{section3}

In \cite{hiaiaudenaert} Audenaert and Hiai determined values of $p$ and $p$ such that the following inequality holds true
$$
\left(\frac{A^p+B^p}{2}\right)^{1/p} \le \left(\frac{A^q+B^q}{2}\right)^{1/q} 
$$
whenever $A$, $B$ are positive semidefinite matrices. When $1/2 \le p \le 1 \le q$, according to the operator convexity of $t^{1/p}$ and operator concavity of $t^{1/q}$ we have
\begin{equation}\label{gen}
\left(\frac{A^p+B^p}{2}\right)^{1/p} \le \frac{A+B}{2} \le \left(\frac{A^q+B^q}{2}\right)^{1/q} 
\end{equation}
Recently, Lam and Le \cite{lam2} studied the quantum divergence generated by these inequalities:
$$
\Phi(A,B)=\Tr \left(\frac{A+B}{2} - \left(\frac{A^p+B^p}{2}\right)^{1/p}\right).
$$

In this section we are going to solve the inverse problem for for $\left(\frac{A^p+B^p}{2}\right)^{1/p}$ and $\left(\frac{A^q+B^q}{2}\right)^{1/q}$. Namely, suppose that $0\le X \le Y$. Solving the inverse mean problem is to find positive definite matrices $A$ and $B$ such that
\begin{equation}\label{inverse}
X = \left(\frac{A^p+B^p}{2}\right)^{1/p}, \quad Y = \left(\frac{A^q+B^q}{2}\right)^{1/q}.
\end{equation}
If this system has a positive solution, then we may use the result to characterize operator monotone function. 

Unfortunately, inequalities in (\ref{gen}) do not characterize operator monotone functions, in general. % Firstly, we have the following lemma, which demonstrates a counter example for the second inequality.

\begin{proposition}\label{counter1}
For any $q>1,$ there exists a non-monotone operator satisfying  
\[ f\left(\dfrac{A+B}{2}\right) \leq f\left(\left(\dfrac{A^q+B^q}{2}\right)^{1/q}\right)\]
for all positive definite matrices $A$ and $B$.
\end{proposition}

\begin{proof}
For a fixed number $ 1 < r \leq \min\{2,q\}$, we consider the operator $f(x) = x^r$ that is a non-monotone operator. It follows from the operator convexity of $x^r$ and the operator concavity of $x^{r/q}$ that 
\[ \left(\dfrac{A+B}{2}\right)^r \leq \dfrac{A^r+B^r}{2} \leq \left(\dfrac{A^q+B^q}{2}\right)^{r/q} \leq \left[\left(\dfrac{A^q+B^q}{2}\right)^{1/q}\right]^{r}.\]
Hence, the operator monotone function $f$ satisfies 
\[ f\left(\dfrac{A+B}{2}\right) \leq f\left(\left(\dfrac{A^q+B^q}{2}\right)^{1/q}\right),\]
which completes the proof of Lemma \ref{counter1}. 
\end{proof}

Similarly, we hope to have the same conclusion for the first inequality, namely, for $1/2 \leq p \leq 1$, there exists a non-monotone operator satisfying
\[ f\left(\left(\dfrac{A^p+B^p}{2}\right)^{1/p}\right) \leq f\left(\dfrac{A+B}{2}\right).\]
However, when $p=1/2$ we could be able to solve the inverse problem and establish a new characterization of operator monotone functions. 

\begin{theorem}\label{Theorem for p=1/2}
Let $f$ be a continuous function on $[0, \infty)$ that satisfies the following inequality
\[ f\left(\left(\dfrac{A^{1/2}+B^{1/2}}{2}\right)^{2}\right) \leq f\left(\dfrac{A+B}{2}\right),\]
for any positive semidefinite matrices $A$ and $B$. Then $f$ is operator monotone. 
\end{theorem}
\begin{proof}
Firstly, we show that $f(X) \leq f(Y)$ for any positive semidefinite matrices $X,Y$ with $0 \leq X \leq Y \leq 2X$. Indeed, we need to solve the following system
\begin{equation}\label{system11}
\begin{cases}
    \left(\dfrac{A^{1/2}+B^{1/2}}{2}\right)^{2} = X \\
    \dfrac{A+B}{2} = Y.
\end{cases} 
\end{equation} 
Subtracting the first equation from the second, we obtain
$$
Y-X = \left(\dfrac{A^{1/2}-B^{1/2}}{2}\right)^{2}.
$$
Therefore, system (\ref{system11}) is equivalent to 
\begin{equation*}
\begin{cases}
    \dfrac{A^{1/2}+B^{1/2}}{2} = X^{1/2} \\
    \dfrac{A^{1/2}-B^{1/2}}{2} =( Y - X)^{1/2}.
\end{cases} 
\end{equation*}
The last system has a unique positive solution as
\begin{align*}
A = \left(X^{1/2} + (Y-X)^{1/2}\right)^2, \quad 
B = \left(X^{1/2} - (Y-X)^{1/2}\right)^2.
\end{align*}
Mention that the condition $0 \leq X \leq Y < 2X$ guarantees the semidefinite positivity of $A$ and $B$. Thus, $f(X) \leq f(Y)$.

Now, for any positive semidefinite matrices $0 \leq X \leq Y$, we identically apply the arguments in the proof of Theorem \ref{Matrix.Theo 2} to obtain a positive integer $m$ and positive semidefinite matrices $Z_1,\dots,Z_m$ such that 
\[ 0 \leq X = Z_1 \leq Z_2 \leq Z_3 \leq \dots \leq Z_{m-1} \leq Y \leq Z_m\]
with $Z_k \leq Z_{k+1} < 2Z_k$ for all $k=1,2,\dots,m$.
Therefore, combining with the previous arguments, we get 
\[ f(X)=f(Z_1) \leq f(Z_2) \leq \dots \leq f\left(Z_{m-1}\right) \leq f(Y).\]
%We completes the proof of Theorem \ref{Theorem for p=1/2}.
\end{proof} 

From the proof of Theorem \ref{Theorem for p=1/2}, under a special condition on $X$ and $Y$,  the inverse problem (\ref{inverse}) has a positive solution.
\begin{corollary}\label{cor}
The following inverse problem has a positive semidefinite solution for $X^2 \le Y^2 \le \sqrt{2} X^2,$  
\begin{equation}\label{cor3}
\frac{A+B}{2} = X, \quad \left(\dfrac{A^2+B^2}{2}\right)^{1/2} =Y.
\end{equation}
\end{corollary}
\begin{proof}
If we put $A_0 = A^2, B_0=B^2$, then the system (\ref{cor3}) is equivalent to the following
\begin{equation*}\label{sys2}
\frac{A_0^{1/2}+B_0^{1/2}}{2} = X, \quad \left(\dfrac{A_0+B_0}{2}\right)^{1/2} =Y,
\end{equation*}
or,
\begin{equation}\label{sys2}
\left(\frac{A_0^{1/2}+B_0^{1/2}}{2}\right)^2 = X^2, \quad \dfrac{A_0+B_0}{2} =Y^2.
\end{equation}
According to the proof of Theorem \ref{Theorem for p=1/2}, the last system has positive semidefinite solutions when $X^2 \le Y^2 \le 2X^2$.  
\end{proof}

The inverse problem for general matrix power means is still unsolved.  

{\bf Open question:} Let $1/2 \le p \le 1 \le q$. Does the following system have positive solutions
\begin{equation*}
\begin{cases}
    \left(\dfrac{A^{p}+B^{p}}{2}\right)^{1/p} = X \\
    \left(\dfrac{A^{q}+B^{q}}{2}\right)^{1/q} = Y
\end{cases} 
\end{equation*} 
whenever $0 \le X \le Y$? This case is just a special case of a general result obtained by Hiai and Audenaert in \cite{hiaiaudenaert}. If the answer is positive, then we immediately have a new characterization: a continuous function $f$ is operator monotone function on $(0, \infty)$ if and only if for any positive definite matrices $A$ and $B$  the following inequality
$$
f\left(\left(\dfrac{A^{p}+B^{p}}{2}\right)^{1/p}\right) \le f\left(\left(\dfrac{A^{q}+B^{q}}{2}\right)^{1/q}\right)
$$
holds true.

\subsection*{Acknowledgement}
We are very grateful to Prof. Mikael de la Salle for useful discussions. The Van Nguyen  was funded by Vingroup Joint Stock Company and supported by the Domestic Master Scholarship Programme of Vingroup Innovation Foundation (VINIF), Vingroup Big Data Institute (VINBIGDATA), code VINIF.2020.ThS.KHTN.05.

\end{document}